\documentclass[12pt]{amsart} 

\usepackage{amsfonts}
\usepackage{amsmath}
\usepackage{color}

\usepackage[OT2,T1]{fontenc}

\usepackage[OT2,T1]{fontenc}
\usepackage{lipsum}

\newcommand\blfootnote[1]{%
  \begingroup
  \renewcommand\thefootnote{}\footnote{#1}%
  \addtocounter{footnote}{-1}%
  \endgroup
}

\usepackage[margin=1.in]{geometry} 
\usepackage{graphicx} 

\newtheorem{thm}{Theorem}[section]
\newtheorem{cor}[thm]{Corollary}
\newtheorem{lem}[thm]{Lemma}

\theoremstyle{definition}
\newtheorem{defn}[thm]{Definition}
\theoremstyle{remark}
\newtheorem{rem}{Remark}
\numberwithin{equation}{section}
\newcommand{\norm}[1]{\left\Vert#1\right\Vert}

\newcommand{\Real}{\mathbb R}

\makeatletter
\DeclareRobustCommand*\cal{\@fontswitch\relax\mathcal}
\makeatother

\newcommand{\RR}{\mathbb{R}}

\newlanguage\fakelanguage
\newcommand\cyr{\fontencoding{OT2}\fontfamily{wncyr}\selectfont
\language\fakelanguage}
\DeclareTextFontCommand{\textcyr}{\cyr}

\title{Stochastic Functional Differential Equations and Feynman-Kac Formula}
\author{Stefano Belloni \textsuperscript{1,2}}

\date{} 

\usepackage{fancyhdr}
\begin{document}

\maketitle


%
%
%
%

\begin{abstract}

In the framework of stochastic functional differential equations (SFDE's) and the corresponding calculus developed in the recent year by F. Yan and S. Mohammed, we provide a series of representation formulae for a variety of highly degenerate functional differential equations of the type of the Feynman-Kac formulae. More precisely, we study the stochastic process satisfying regular SFDE's with killing and absorbing boundary, we give the differential equation to be solved in order to compute the distribution of the first exit time from a regular domain and apply our results to a model describing bacterial motility and to the derivation of a path-dependent Black-Scholes equation. 

\end{abstract}

\blfootnote{
University of Heidelberg\\
Im Neuenheimer Feld 267\\
69120 Heidelberg\\
\emph{Email address:} stefano.belloni@bioquant.uni-heidelberg.de\\\\
1- {Center for Modeling and Simulation in the Biosciences (BIOMS)}\\
2- {Interdisciplinary Center for Scientific Computing (IWR)}}

\section{Introduction}

The present article deals with some extensions of the basic Feynman-Kac representation formula stated by Yan and Mohammed in \cite{YM05} in the framework of functional stochastic differential equations (SFDE).\\ 
The mathematical literature is rich in generalizations and variations of the Feynman-Kac formula. Following \cite{MF85} the Feynman-Kac formula represents the solution of the problem:
\begin{equation*}
  \begin{cases}
    \frac{\partial u}{\partial t}(t,x)+L_tu(t,x)+c(t,x)u(t,x)=f(t,x), &
    (t,x)\in[0,T)\times\RR^d \\
    u(T,x)=g(x), & x\in\RR^d
  \end{cases}
\end{equation*}
in terms of the diffusion process corresponding to the operator $L_t$, given by
$$L_tu(t,x)=\frac{1}{2}\sum_{i,j=1}^da_{i,j}(t,x)\frac{\partial^2}{\partial{x_i}\partial{x_j}}u(t,x)+\sum_{i=1}^d b(t,x)\frac{\partial}{\partial x_{i}}u(t,x)$$
which is intimately connected with It\^o's Lemma and the martingale problem for SDE (see \cite{VS79}).\\
We consider a degenerate linear functional partial differential equation (FPDE):
\begin{align*}
\partial_t u(t,x,\eta)=&Su(t,x,\eta)+{D_xu(t,x,\eta)}(H(t,x,\eta)) +c(x,\eta)u(t,x,\eta)\\
&+\frac{1}{2}\sum_{j=1}^m{D^2_{xx}u(t,x,\eta)}(G(t,x,\eta)(\mathbf{e}_j)\otimes G(t,x\eta)(\mathbf{e}_j))
\end{align*}
where $(e_j)_{j=1\ldots m}$ is a normalized base of $\Real^m$, $(t,x,\eta)\in[0,T]\times\Real^d\times L^2([-r,0],\Real^d)$ and $S$ is the \emph{shift operator} (see Section 2).We study the associated Cauchy's problem and, given a domain in $\Real^d\times L^2([-r,0],\Real^d)$, the corresponding Dirichlet and mixed problem. We derive representation formulae in terms of the average over the paths of a functional of the related SFDE's.\\

In the study and formulation of an SFDE, the segment process of a continuous-time stochastic process, defined for $t\in[0,T]$ and $s\in[-r,0]$ as $X_t:=X(t+s)$, is an important ingredient. Yan and Mohammed  developed the above stochastic segment integral and its calculus \cite{YM05}. Although it is possible to define the stochastic integral with respect to an infinite dimensional martingale (\cite{SAV84,PZ92}), one cannot apply this definition to the Brownian segment process because it is not a $L^2([-r,0];\Real^d)$-valued martingale. With the help of anticipating stochastic calculus \cite{NP88} they defined the segment integral and proven It\^{o}'s formula for segments of solutions of SFDE's.\\

In the setting of classical PDE's, the Feynman-Kac formula for SDE's whose coefficients do not depended on the history of the process, also represents a link between functionals of the diffusion process and PDEs, for example the exit time of a diffusion process from a domain \cite{MF85}.\\
The distribution of the first exit times for random processes are key quantities in many scientific fields, such as mathematical physics, neuroscience, economics or mathematical finance. Unfortunately, closed form
solutions to this problem are not attainable except in a few special cases \cite{PW08,MF85}. With the help of the Feynman-Kac formula, we derive which is the FPDE that the first exit times distribution satisfies (as a viscosity solution).\\

The article is organized as follows. in Section 2 the main definitions are presented, and a summary of the main results from \cite{NP88,YM05} is provided. In Section 3 the Feynman-Kac representation formula is proven and we discuss the reverse result in terms of viscosity solutions. Section 4 deals with the Feynman-Kac formula for a domain in $\Real^d\times L^2([-r,0],\Real^d)$. In Section 5 we present the problem of the first exit time distribution for SFDE's. Finally, in Section 6, we present some applications of the formula.


\section{Main Definitions and Notations}
\subsection{SFDE and Weak Infinitesimal Generator}

In the present section we summarize the results in \cite{NP88,YM05} about SFDE's and relative Functional It\^{o} calculus.
Let consider the following $d$-dimensional SFDE with bounded memory $r\in[0,\infty)$:
\begin{equation}\label{Eq_1}
dX(s)=H(s,X_s)ds+G(s,X_s)dW(s),\quad s\in[0,T]
\end{equation}
with the initial condition
$$(t,X_t)=(t,X_0,\eta_t)\in[0,T]\times L^2(\Omega,\Real^d;\mathcal{F}(t)) \times L^2(\Omega,L^2([-r,0],\Real^d);\mathcal{F}(t)),$$
where $X_t$ is the segment process defined by
$$X_t(s):=X(t+s)\quad s\in[-r,0],$$ 
and the coefficients
\begin{align*}
H:&[0,T]\times L^2(\Omega,\Real^d\times L^2([-r,0],\Real^d))\to\Real^d\\
G:&[0,T]\times L^2(\Omega,\Real^d\times L^2([-r,0],\Real^d))\to\Real^{d\times n}.
\end{align*}
satisfy the standard assumption of uniform Lipschitz continuity with respect to the variable $t$, i.e.
The functions $H$ and $G$ are continuous functions that satisfy the following Lipschitz continuity.
\textbf{Assumption 1} (\emph{Lipschitz Continuity}) There exists a constant $K_{L}>0$ such that for all $(s,\phi)$ and $(t,\psi)\in[0,T]\times L^2(\Omega,\mathbf{C})$
$$\mathbb{E}\Big[|H(s,\phi)-H(t,\psi)|+|G(s,\phi)-G(t,\psi)|\Big]\leq K_L\Big(|t-s|+\norm{\phi-\psi}_{L^2(\Omega,\mathbf{C})}\Big).$$
We suppose that $F$ and $G$ satisfy either\\ 
\textbf{Assumption 2} There exists a constant $K_G>0$ such that
$$\mathbb{E}\Big[|H(s,\phi)|+|G(s,\phi)|\Big]\leq K_G\Big(1+\mathbb{E}[\norm{\phi}]\Big),$$
or the monotonicity one:\\
\textbf{Assumption 3} For each compact subset $C\subset\mathbf{C}$, there exists a number $K_C$ and some
$r_C\in(0, r)$ such that for all $x, y \in \mathbf{C}$ with $x(s) = y(s)$ for all $s\in [r,r_C]$
$$2\langle H(x)-H(y), x(0)-y(0)\rangle + \norm{G(x)-G(y)}^2\leq K_C\norm{x - y}^2 .$$

Under these hypotheses there exists a strong and unique solution of the equation in the following sense (see \cite{A97}):

\begin{thm}
Suppose that \textbf{Assumption 1} and \textbf{Assumption 2} (or \textbf{Assumption 3}) hold. Then the SFDE (\ref{Eq_1}) has a unique strong solution $\{X(s;t,\psi_t),s\in[t-r,T]\}$, i.e.$X\in\mathcal{L}^2(\Omega,\mathbf{C}([t-r,T]),\Real^m)$ adapted to the filtration generated by the Bownian motion and unique up to equivalence in $\mathcal{L}^2(\Omega,\mathbf{C}([t-r,T]),\Real^m)$.
\end{thm}

Following the theory developed by Mohammed on SFDE's \cite{MohammedSFDE,A97}, it is possible to associate to (\ref{Eq_1}) a Markov process and a Markov family 
$$\Big\{( ^\eta X_t,^\eta X_t(0)):t\in[0,T],(\eta(0),\eta)\in \Real^d\times L^2([-r,0],\Real^d)\Big\}$$
on $\Real^d\times L^2([-r,0],\Real^d)$
with transition probability
$$p(t_1,\eta,\eta(0);t_2,B_C,B_{\Real^d})=\mathbf{P}\Big\{( ^\eta X_{t_2}, ^\eta X_{t_2}(0))\in B_{ L^2([-r,0],\Real^d)}\times B_{\Real^d} \Big| \mathcal{F}_{t_1}\Big\},$$
where $B_{ H}$ is a Borel set of $H$, that satisfies the condition:
$$\mathbf{P}\Big\{( ^\eta X_{t_2}, ^\eta X_{t_2}(0))\in B_{ L^2([-r,0],\Real^d)}\times B_{\Real^d} | \mathcal{F}_{t_1}\Big\}$$
$$=\mathbf{P}\Big\{( ^\eta X_{t_2}, ^\eta X_{t_2}(0))\in B_{ L^2([-r,0],\Real^d)}\times B_{\Real^d} | (^\eta X_{t_1}(0),^\eta X_{t_1})\Big\}.$$
For a Borel measurable function $\Phi:\Real^d\times L^2([-r,0],\Real^d)\to\Real$, we also define the Shift Operator
$$\Gamma_t(\Phi)(v,\phi):=\Phi(v,\tilde{\phi_t})$$
where for each $(v,\phi)\in \Real^d\times L^2([-r,0],\Real^d)$ and $t\geq0$ $\tilde\phi:[-r,\infty)\to\Real^n$ is defined by
$$\tilde\phi(t):=
\begin{cases}
v&t>0\\
\phi(t)&t\in[-r,0].
\end{cases}
$$
Let us define the operator
$$S(\Phi)(\phi):=\lim_{t\to0}\frac{1}{t}\Big[\Gamma_t(\Phi)(v,\phi)-\Phi(v,\phi)\Big]$$
whose domain $\mathcal{D}(S)$ is defined as the set of functions for which the limit exists.
We can study the infinitesimal generator of the SFDE \cite{A97}:
\begin{thm}\label{Mohammedgenerator}
Suppose that $\Phi:[0,T]\times \Real^d\times L^2([-r,0],\Real)\to\Real$ belongs to the class $\mathbf{C}_b$ of bounded uniformly continuous functions with the supremum norm $\norm{\Phi}_{\mathbf{C}_b}:=\sup\{\Phi(t,v,\theta):(t,v,\theta)\in[0,T]\times\Real^d\times L^2([-r,0],\Real^d)\}$, and satisfies the following conditions:
\begin{itemize}
\item $\Phi\in\mathcal{D}(\Gamma)$.
\item $\Phi\in\mathbf{C}^2$.
\item $D\Phi$ and $D^2\Phi$ are globally bounded.
\item $D^2\Phi$ is globally Lipschitz on $\Real^d\times L^2([-r,0],\Real^d)$ uniformly with respect to $t$.
\end{itemize}
Then $\Phi\in\mathcal{D}(\mathcal{A}_w)$ and for each $(t,v,\phi)\in [0,T]\times\Real^d\times L^2([-r,0],\Real^d)$
\begin{align*}
\mathcal{A}_w\Phi(t,\phi)=&\lim_{\epsilon\to0}\frac{\mathbb{E}[\Phi(t+\epsilon,X(t+\epsilon),X_{t+\epsilon})]-\Phi(t,v,\phi_t)}{\epsilon}\\
=& \frac{\partial }{\partial t}\Phi(t,v,\phi)+S(\Phi)(t,v,\phi)+D_1\Phi(t,v,\phi)(H(t,v,\phi))\\
&+\frac{1}{2}\sum_{j=1}^mD_{11}^2\Phi(t,v,\phi)(G(t,v,\phi)(\mathbf{e}_j),G(t,\phi)(\mathbf{e}_j))
\end{align*}
where $\mathbf{e}_j, j=1\ldots n$ is the $j-$th vector of the standard basis in $\Real^m$ 
\end{thm}

%

\subsection{Anticipating Calculus and It\^o's Formula for SFDE}

The It\^{o} formula derived by Mohammed and Yan for processes solutions of SFDE is proven via anticipating calculus methods. To understand the need for anticipating calculus in such an intrinsically adapted setting, it is instructive to look at the following simple one-dimensional SDDE, where $g$ is a regular function:
$$dX(t)=g(X(t-1),X(t))dW(t), t\leq 0$$ $$X(t)=W(t), t\in[-1,0].$$
Formally for $t\in(0,1]$
\begin{align*}
dg(X(t-1),X(t))=&dg(W(t-1),X(t))\\
=&\frac{\partial g}{\partial x}(W(t -1),X(t)) dW(t - 1)\\
&+\frac{\partial g}{\partial y}(W(t - 1),X(t))g(X(t-1),X(t))dW(t)\\
&+\textrm{ second order terms}
\end{align*}
Note that although the coefficient $g(X(t -1),X(t))$ is $\mathcal F_t$ -measurable, the first
term $\frac{\partial g}{\partial x}(W(t -1),X(t)) dW(t - 1)$ on the right-hand side of the last equality
is an anticipating differential.

We denote by $D$ the Malliavin differentiation operator. Let $F$ be a random variable which belongs to the domain of $D$ and $T=[0,T]$. Its derivative $DF$ is a stochastic process $\{D_tF:t\in T\}$. The derivative $DF$ may be considered as a random variable taking value in the Hilbert space $H=L^2(T,\Real^n)$. More generally the Nth derivative of $F$, $D^NF:=D^{j_1}_{s_1}\cdots D^{j_N}_{s_N}$ is an $H^{\hat\otimes_2 N}$ random variable. For any positive integer $N$ and real number $p>1$ we denote by $\mathbb{D}^{N,p}$ the Banach space of all the random variables having all the $i$-th derivatives belonging to $L^p(\Omega,H^{\hat\otimes_2 N})$ with the norm defined by
$$\norm{F}_{N,p}=\norm{F}+\norm{\norm{D^NF}_{(2)}}_{p},$$
where $\norm{\cdot}_{(2)}$ is the Hilbert-Schmidt norm in $H^{\hat\otimes_2 N}$
$$\norm{D^NF}^2_{(2)}=\sum_{j_1,\cdot j_N=1}^n\int_{T^N}\mathbb{E}[(D^NF)^{j_1,\ldots,j_N}_{s_1,\ldots,s_N}]^2ds_1\ldots ds_N$$

We denote by $\delta$ the divergence operator, and by $\delta(u)$ the Skorohod stochastic integral of the process $u$. $\delta$ is the adjoint operator of $D$. We denote by $\mathbb{L}^{1,2}$ the class of all processes $u\in L^2(T\times\Omega)$ such that $u(t)\in\mathbb{D}^{1,2}$ for almost all $t$ and there exist a measurable version of the process $D_s u(t)$ (which depends on two parameter) satisfying $E\int_T\int_T(D_su(t))^2dsdt<\infty$. $\mathbb{L}^{1,2}$ is a Hilbert space with the norm
$$\norm{u}^2_{1,2}=\norm{u}^2_{L^2(T\times\Omega)}+\norm{Du}^2_{L^2(T^2\times\Omega)}.$$
Note that $\mathbb{L}^{1,2}$ is isomorphic to $L^2(\Omega,\mathbb{D}^{1,2})$. For every $p>1$ and any positive integer $k$ we denote by $\mathbb{L}^{k,p}$ the space $L^2(\Omega,\mathbb{D}^{k,p}).$\\\\
Now let us define the \emph{segment operator} $\mathcal{O}:H\oplus V\to H\hat{\otimes}_2 V$
$$\mathcal{O}\phi(t,s):=\phi(t+s),\quad t\in[-r,0],s\in[0,T]\quad \phi\in H\oplus V$$
Let $\mathcal{O}_t\phi=\phi_t$ and $\mathcal{O}^\ast:H\hat{\otimes}_2 V\to H\oplus V$ the adjoint.
Denote by $P_H$ (resp $P_V$) the projection from $H\oplus V$ on $H$ (resp. $V$), and define $\mathcal{O}^\ast_H=P_H\circ \mathcal{O}^\ast$ 

\begin{defn} Suppose $W = (W(t))_{t\in[0,T]}$ is a $m$-dimensional standard Brownian motion. Denote
by $\delta$ the divergence operator and $Dom(\delta)$ its domain. For a two parameter process $X\in(\mathcal{O}_H^\ast)^{-1}(Dom(\delta))$ the Skorohod segment integral of X with respect to the Brownian segment $(W_t)_{t\in[0,T]}$ is defined as
$$\int_0^T\langle X_t,W_t\rangle=\delta(\mathcal{O}_H^\ast X)$$
\end{defn}
\subsubsection{It\^{o} Formula} Consider the SFDE:
\begin{equation*}
X(t)=\tilde\eta_0(t)+\int_0^{t} v(s)ds+\int_0^{t} u(s)dW(s),\quad t\geq 0
\end{equation*}
with initial condition $\tilde\eta\in L^2([-r,0],\Real^d)$ and coefficients $u:T\times\Omega\to L(\Real^n,\Real^m)$ and $v:T\times\Omega\to \Real^m$ that may not be adapted to the Brownian filtration $(\mathcal{F}_t)_{t\geq0}$.

\begin{thm}[It\^{o}'s Formula]\label{ItoFormula} Let $f=f(t,\eta,x)\in\mathbf{C}^1_b(T\times V\times \Real^m)$ with second bounded derivative, $u\in\mathbb{L}^{1,2}$ and $v\in\mathbb{L}^{1,2}$. Then the following It\^{o} formula holds:
\begin{align*}
f(t,X_t,X(t))=&\,\,f(0,X_0,X(0))\\&+\int_0^t\frac{\partial f}{\partial s}(s,X_s,X(s))ds+\int_0^t\langle\frac{\partial f}{\partial\eta}(s,X_s,X(s),dX_s\rangle_V\\
&+\int_0^t\frac{\partial f}{\partial x}(s,X_s,S(s))dX(s)
+\int_0^t \frac{\partial^2f}{\partial\eta^2}(s,X_s,X(s))(\Theta_s)ds
\\
&+\int_0^t\frac{\partial^2 f}{\partial\eta\partial x}(s,X_s,X(s))[(u\Lambda)_sX(s)]ds\\&+\int_0^t\frac{\partial^2}{\partial x\partial\eta}(s,X_s,X(s))[u(s)D_sX_s]ds\\
&+\frac{1}{2}\sum_{i=1}^d\int_0^t\frac{\partial^2f}{\partial x^2}(s,X_s,X(s))[(\nabla_+^iX)(s)\otimes u^{\dot i}(s)]ds
\end{align*} 
Where
$$\Theta_s(\alpha,\beta)=\frac{1}{2}((u\Lambda)_sX_s(\alpha,\beta)+(u\Lambda)_sX_s(\beta,\alpha))$$
$$(u\Lambda)_sX_s(\alpha,\beta)=\mathbf{I}_{\{0\leq s+\alpha\wedge\beta\}}u(s+\alpha)D_{s+\alpha}X(s+\beta)$$
$$(\nabla_+^iX)(s)=\lim_{\epsilon\to0}(D_t^iX(t+\epsilon)+D_t^iX(t-\epsilon))$$
\begin{equation*}
(u\Lambda)_sX(s)(\alpha):=u(s+\alpha)D_{s+\alpha}X(s)\mathbf{I}_{\{s+\alpha\\geq0}
\end{equation*}
\end{thm}

%


\section{The Feynman-Kac Formula}

In this section we prove the Feynman-Kac Formula for SFDE and extend the results in \cite{YM05} and \cite{CY}.\\
Let us consider the following autonomous SFDE:
\begin{equation}\label{C08-2.28}
dX(s)=H(X(s),X_s)ds+G(X(s),X_s)dW(s),\quad s\in[0,T]
\end{equation}
with the initial datum $\eta\in L^2(\Omega,L^2([-r,0],\Real^d)),$ at time $t=0$.\\\\
\textbf{Assumption $A$}: $H\in\mathbb{L}^{1,2}$ and $G\in\mathbb{L}^{1,2}$ are adapted functions that satisfy the hypothesis of Lipschitz continuity with respect to both arguments.\\
\begin{rem}
We stress that \textbf{Assumption A} guarantees existence and uniqueness of the solution of equation (\ref{Eq_1})  (see \cite{MohammedSFDE}).
\end{rem}

\subsection{Representation Formula}

The following theorem extends Theorem 9.5 in \cite{YM05}.

\begin{thm}\label{FeynmannkacStef}
Suppose $f\in\mathcal{D}(\mathcal{A}_w)$ and $c:L^2([-r,0])\times\Real^d\to\Real_+$ bounded and Lipschitz continuous. If $u$ solves weakly
$$\frac{\partial}{\partial t}u(t)+\mathcal{A}_w(u(t))+c\cdot u(t)=0$$
$$u(T,\eta,x)=f(\eta,x)$$
with $\mathcal{A}_w$ as defined in Theorem \ref{Mohammedgenerator}, then 
$$u(t,\eta,x):=\mathbb{E}_{(t,\eta,x)}\Big[f( ^\eta X_T,  ^\eta X_T(0))e^{\int_t^T c( ^\eta X_s,  ^\eta X_s(0))ds}\Big].$$
\end{thm}

\begin{proof}
Let us suppose that $u$ is a solution of the above FPDE.\\

Fix $0\leq t_0<T$. Define for all $t_0\leq t\leq T$
$$q(t):=\mathbf{E}\Big[u( t, X_t,X_t(0))e^{\int_{t_0}^t c( X_s,  X_s(0))ds}\|\mathcal{F}_{t_0}\Big].$$
Now we calculate the right derivative (if it exists) of $q(t)$.\\
Since the process $(X_t,X(t))$ is a Markov process, the following equality holds:
 \begin{align*}
q(t_2)-q(t_1)&=\mathbf{E}\Big\{\mathbf{E}\Big[u( t_2, X_{t_2},X_{t_2}(0))e^{\int_{t_0}^{t_2} c( X_s,  X_s(0))ds}\\
&\qquad\qquad-u( t_1, X_{t_1},X_{t_1}(0))e^{\int_{t_0}^{t_1} c( X_s,  X_s(0))ds}\|\mathcal{F}_{t_1}\Big]\|\mathcal{F}_{t_0}\Big\}\\
&=\mathbf{E}\Big\{\mathbf{E}\Big[u( t_2, X_{t_2},X_{t_2}(0))e^{\int_{t_1}^{t_2} c( X_s,  X_s(0))ds}-u( t_1, X_{t_1},X_{t_1}(0))\|\mathcal{F}_{t_1}\Big]\\
&\qquad\qquad\cdot e^{\int_{t_0}^{t_1} c( X_s,  X_s(0))ds}\|\mathcal{F}_{t_0}\Big\}
\end{align*}
Consider the process 
$$Y(t)=e^{\int_{t_1}^{t}c(X_s,X_s(0))ds},$$
solution of the following stochastic integral equation:
$$Y(t)=1+\int_{t_1}^{t}c(X_s,X_s(0))Y(s)ds.$$
The idea is to apply now the It\^{o}-Mohammed-Yan formula (Theorem \ref{ItoFormula}) to the process
$$h(t,X_t,X_t(0),Y(t))=u(t,X_t,X(0))\cdot Y(t)$$ 
and calculate explicitly the right derivative of the projection of the random variable $h(t)$ on the space $L^2(\mathcal{F}_{t_1})$.
We have that the processes involved are adapted (see \cite{MohammedSFDE}), in particular for the process $Y$, $D_s Y(\alpha)=0$. A straightforward use of the formula leads to
\begin{align*}
&h(t,X_{t},X_{t}(0),Y(t))=h(t_1,X_{t_1},X_{t_1}(0),Y(t_1))=\\
(i)&=\int_{t_1}^t c(X_s,X_s(0))u(s,X_s,X(s))e^{\int_{t_1}^{s}c(X_u,X_u(0))du}ds\\
(ii)&+\int_{t_1}^t\frac{\partial u}{\partial s}(s,X_s,X(s))e^{\int_{t_1}^{s}c(X_u,X_u(0))du}ds\\
(iii)&+\int_{t_1}^t\langle\frac{\partial u}{\partial\eta}(s,X_s,X(s))e^{\int_{t_1}^{s}c(X_u,X_u(0))du},dX_s\rangle_V\\
(iv)&+\int_{t_1}^t\frac{\partial u}{\partial x}(s,X_s,S(s))e^{\int_{t_1}^{s}c(X_u,X_u(0))du}dX(s)
\\
(v)&+\int_{t_1}^t \frac{\partial^2 u}{\partial\eta^2}(s,X_s,X(s))(\Theta_s)e^{\int_{t_1}^{s}c(X_u,X_u(0))du}ds\\
(vi)&+\int_{t_1}^t\frac{\partial^2 u}{\partial\eta\partial x}(s,X_s,X(s))[(G\Lambda)_sX(s)]e^{\int_{t_1}^{s}c(X_u,X_u(0))du}ds\\
(vii)&+\int_{t_1}^t\frac{\partial^2 u}{\partial x\partial\eta}(s,X_s,X(s))[G(s)D_sX_s]e^{\int_{t_1}^{s}c(X_u,X_u(0))du}ds\\
(viii)&+\frac{1}{2}\sum_{i=1}^d\int_{t_1}^t e^{\int_{t_1}^{s}c(X_u,X_u(0))du}\frac{\partial^2 u}{\partial x^2}(s,X_s,X(s))[(\nabla_+^iX)(s)\otimes G^{\dot i}(s)]ds
\end{align*} 
where
$$\Theta_s(\alpha,\beta)=\frac{1}{2}((G\Lambda)_sX_s(\alpha,\beta)+(G\Lambda)_sX_s(\beta,\alpha))$$
$$(G\Lambda)_sX_s(\alpha,\beta)=\mathbf{I}_{\{0\leq s+\alpha\wedge\beta\}}G(s+\alpha)D_{s+\alpha}X(s+\beta)$$
$$(\nabla_+^i X)(s)=\lim_{\epsilon\to0}(D_t^iX(t+\epsilon)+D_t^iX(t-\epsilon))$$

We treat now every single term (identified by the Roman number ($\alpha$)), calculating the limit for $t$ approaching $t_1$ of the quantity $\frac{1}{t-t_1}\cdot \mathbb{E}[(\alpha)\|\mathcal{F}_{t_1}]$.\\
Since we consider the limit of the projection on the $\sigma$-algebra $\mathcal{F}_{t_1}$, we consider for $t>t_1$ the SFDE with $B$ a Brownian motion s.t. $B(t)=0$ in $[-r,t_1]$ and $Z(t)=X(t)$, 
$$Z(t)=X_{t_1}(0)+\int_{t_1}^{t\vee t_1} H(Z_s,Z(s))ds+\int_{t_1}^{t\vee t_1} G(Z_s,Z(s))dB(s)$$
so that $\|\Theta_s\|_{(V\otimes V)^\ast}\to0$.\\ 
In this case it follows that the addends (v),(vi) and (vii) converge to 0.\\
A straightforward calculation for the terms (i), (ii), (iv), using the boundedness and Lipschitz continuity of the function $c$, leads to:
\begin{itemize}
\item for the term (i):
$$\lim_{t_2\searrow t_1}\mathbf{E}\Big[\frac{1}{t_2-t_1}\int_{t_1}^{t_2} c(X_s,X_s(0))u(s,X_s,X(s))e^{\int_{t_1}^{s}c(X_u,X_u(0))du}\|\mathcal{F}_{t_1}\Big]$$ $$= c(X_{t_1},X_{t_1}(0))u(t_1,X_{t_1},X(t_1));$$
\item for the term (ii):
$$\lim_{t_2\searrow t_1}\mathbf{E}\Big[\frac{1}{t_2-t_1}\int_{t_1}^{t_2} \frac{\partial u}{\partial s}(s,X_s,X(s))e^{\int_{t_1}^{s}c(X_u,X_u(0))du}ds\|\mathcal{F}_{t_1}\Big]$$ $$= \frac{\partial u}{\partial t}(t_1,X_{t_1},X(t_1));$$
\item for the term (iv):
\begin{align*}
\lim_{t_2\searrow t_1}\mathbf{E}\Big[&\frac{1}{t_2-t_1}\int_{t_1}^{t_2} \frac{\partial u}{\partial x}(s,X_s,S(s))e^{\int_{t_1}^{s}c(X_u,X_u(0))du}dX(s)\|\mathcal{F}_{t_1}\Big]\\
&=\lim_{t_2\searrow t_1}\mathbf{E}\Big[\frac{1}{t_2-t_1}\int_{t_1}^{t_2} \frac{\partial u}{\partial x}(s,X_s,S(s))e^{\int_{t_1}^{s}c(X_u,X_u(0))du}\\
&\quad\quad\quad\quad\quad\quad \cdot H(X_s,X_u(s))ds\|\mathcal{F}_{t_1}\Big]+0\\
&=H(X_{t_1},X(t_1))\frac{\partial u}{\partial x}(t_1,X_{t_1},X(t_1))
\end{align*} 
\item We concentrate our attention to the third term (iii).
\end{itemize}
This term contains the integral with respect to the \emph{segment process} (we refer to the Appendix for its definition and properties).\\
We have to check that for the function $$g(t):=\int_{t_0}^{t}\langle\frac{\partial u}{\partial\eta}(s,X_s,X(s)e^{\int_{0}^{t}c(X_u,X_u(0))du},dX_s\rangle_V$$ holds:
$$\lim_{t_2\searrow t_1}\mathbf{E}\Big[\frac{g(t_2)-g(t_1)}{t_2-t_1}\|\mathcal{F}_{t_1}\Big]=\langle\frac{\partial u}{\partial\eta}(t_1,X_{t_1},X_{t_1}(0),dX_{t_1}\rangle_V\cdot e^{\int_{t_0}^{t_1}c(X_u,X_u(0))du}$$
By taking into account that $c$ is Lipschitz continuous and bounded, by the stochastic Fubini's theorem (Lemma 4.2 in \cite{YM05}) and the definition of segment integral, we compute the following limit:
$$
\lim_{t_2\searrow t_1}\mathbf{E}\Big[\frac{g(t_2)-g(t_1)}{t_2-t_1}\|\mathcal{F}_{t_1}\Big]=(\ast)$$
$$\lim_{t_2\searrow t_1}\mathbf{E}\Big[\int_r^0\int_{t_1}^{t_2}e^{\int_{t_1}^{t_2}c(X_u,X_u(0))du}\frac{\partial u}{\partial\eta}(s,X_s,X(s))(\alpha)dX(\alpha+s)d\alpha\Big\|\mathcal{F}_{t_1}\Big]e^{\int_{t_0}^{t_1}c(X_u,X_u(0))du}
$$
\normalsize
We omit in what follows the factor $e^{\int_{t_0}^{t_1}c(X_u,X_u(0))du}$.
\begin{align*}
(\ast)&=\lim_{t_2\searrow t_1}\mathbf{E}\Big[\int_r^0\frac{1}{t_2-t_1}\int_{t_1}^{t_2}e^{\int_{t_1}^{s}c(X_u,X_u(0))du}\\
&\quad\quad\quad\quad\quad\quad\cdot\frac{\partial u}{\partial\eta}(s,X_s,X(s))(\alpha)\mathbf{1}_{\{s+\alpha\geq t_1\}}dX(\alpha+s)d\alpha\Big\|\mathcal{F}_{t_1}\Big]\\
&+\lim_{t_2\searrow t_1}\mathbf{E}\Big[\int_r^0\frac{1}{t_2-t_1}\int_{t_1}^{t_2}e^{\int_{t_1}^{s}c(X_u,X_u(0))du}\\
&\quad\quad\quad\quad\quad\quad\cdot\frac{\partial u}{\partial\eta}(s,X_s,X(s))(\alpha)\mathbf{1}_{\{s+\alpha> t_1\}}H(\alpha+s)dsd\alpha\Big\|\mathcal{F}_{t_1}\Big]
\end{align*}
\begin{align*}
&+\lim_{t_2\searrow t_1}\mathbf{E}\Big[\int_r^0\frac{1}{t_2-t_1}\int_{t_1}^{t_2}e^{\int_{s}^{t_2}c(X_u,X_u(0))du}\\
&\quad\quad\quad\quad\quad\quad\cdot\frac{\partial u}{\partial\eta}(s,X_s,X(s))(\alpha)\mathbf{1}_{\{s+\alpha> t_1\}}G(s+\alpha)dW(\alpha+s)d\alpha\Big\|\mathcal{F}_{t_1}\Big]\\
&=\langle\frac{\partial u}{\partial\eta}(t_1,X_{t_1},X_{t_1}(0),dX_{t_1}\rangle_V=Su(t_1,X_{t_1},X_{t_1}(0)
\end{align*} 
since the second and third integral are 0.
\normalsize
From hypothesis $$\Big(\frac{\partial}{\partial t}+\mathcal{A}_w+c\cdot\mathbb{I}\Big)u( t, \eta,\eta_{t_1}(0))=0$$ we can conclude that on $\mathcal{F}_{t_0}$
$$\lim_{t_2\searrow t_1}\frac{q(t_2)-q(t_1)}{t_2-t_1}=$$
$$=\mathbf{E}\Big\{\mathbf{E}\Big[\Big(\frac{\partial}{\partial t}+\mathcal{A}_w+c\cdot\mathbb{I}\Big)u( t_1, X_{t_1},X_{t_1}(0))\|\mathcal{F}_{t_1}\Big]e^{\int_{t_0}^{t_1} c( X_s,  X_s(0))ds}\|\mathcal{F}_{t_0}\Big\}\equiv 0.$$
Thus the function $q$ is continuous and has continuous right derivatives. By a well-known Lemma (\cite{Yosida95}, p. 239), $q$ is differentiable and hence a constant. We conclude that
\begin{align*}
q(t_0)=q(T)&=\mathbf{E}\Big[u( T, X_T,X(T))e^{\int_{t_0}^T c( X_s,  X_s(0))ds}\|\mathcal{F}_{t_0}\Big]\\
&=\mathbf{E}\Big[f( X_T,X(T))e^{\int_{t_0}^T c( X_s,  X_s(0))ds}\|\mathcal{F}_{t_0}\Big]
\end{align*}

\end{proof}

This theorem can generalize to the case where $c$ is time in-homogenous by using the same approach:
\begin{cor}Suppose $f\in\mathcal{D}(\mathcal{A}_w)$ and $c:[-r,T]\times L^2([-r,0])\times\Real^d\to\Real_+$ bounded and (\emph{maybe} Lipschitz) continuous: if $u$ solves weakly
$$\frac{\partial}{\partial t}u(t)+\mathcal{A}_w(u(t))+c(t,\cdot)\cdot u(t)=0$$
$$u(T,\eta,x)=f(\eta,x),$$
with $\mathcal{A}_w$ as defined in Theorem \ref{Mohammedgenerator}, then 
$$u(t,\eta,x):=\mathbb{E}_{(t,\eta,x)}\Big[f( ^\eta X_T,  ^\eta X_T(0))e^{\int_t^T c(s, ^\eta X_s,  ^\eta X_s(0))ds}\Big]$$
\end{cor}

In what follows, we assume that \textbf{Assumption A} is satisfied. Consider the following SFDE:\\
Set $\eta_0:[-r,t]\to\Real^d$, with $\eta_0(s)=\eta_0(0)$ for $s\geq 0$ and $t^{t,\eta}_s=t-s$. 
If $s\in[-r,t]$ 
\begin{equation*}
X^{t,\eta}(s)=\eta_0(s)+\int_0^{s\vee0} H(t^{t,\eta}_s,X^{t,\eta}_u,X^{t,\eta}(u))ds+\int_0^{s\vee0} G(t^{t,\eta}_s,X^{t,\eta}_u,X^{t,\eta}(u))dW(u).
\end{equation*}
It is possible to define in accordance with the previous section a Markov family $$\Big(t^{t,\eta}_s, ^\eta X_s, ^\eta X_s(0)\Big)\in\Real\times L^2([-r,0],\Real^d)\times\Real^d$$
In this case the infinitesimal generator is given by $\mathcal{A}^-_w$, defined as
\begin{align*}
\widetilde{\mathcal{A}}^-_w\Phi(t,\phi)=& -\frac{\partial }{\partial t}\Phi(t,\phi_t)+S(\Phi)(t,\phi_t)+\overline{D\Phi(t,\phi_t}(H(t,\phi_t)\mathbf{1}_{\{0\}})\\
&+\frac{1}{2}\sum_{j=1}^m\overline{D^2\Phi(t,\phi_t}(G(t,\phi_t)(\mathbf{e}_j)\mathbf{1}_{\{0\}},G(t,\phi_t)(\mathbf{e}_j)\mathbf{1}_{\{0\}})
\end{align*}
If the coefficients are homogenous in time, then we have the following result: 
\begin{cor}Suppose $f\in\mathcal{D}(\mathcal{A}^-_w)$. If $u$ solves weakly
$$\frac{\partial}{\partial t}u(t)=\widetilde{\mathcal{A}}^-_w(u(t))$$
$$u(0,\eta,x)=f(\eta,x)$$
then
$$u(t,\eta,x)=\mathbb{E}_{(0,\eta,x)}\Big[f( ^\eta X_t,  ^\eta X_t(0))\Big]$$
\end{cor}
\begin{proof} Let $u(t,\eta,x)$ be a solution of $\frac{\partial}{\partial t}u(t)=\mathcal{A}_w(u(t))$, $u(0,\eta,x)=f(\eta,x)$. Consider now, for a fixed but arbitrary $T$, the function $v(t,\eta,x):=u(T-t,\eta,x)$. Let us consider now the random variable
$$q(t):=\mathbb{E}[v(t,X_t,X(t))\|\mathcal{F}_0]$$
exactly as before $\frac{d}{dt^+}q(t)=0$ on $\mathcal{F}_0$ so one has the equalities:
\begin{align*}
q(0)=q(T)&=\mathbf{E}\Big[v( T, X_T,X(T))\|\mathcal{F}_{0}\Big]\\
&=\mathbf{E}\Big[v(0, X_0,X(0))\|\mathcal{F}_{0}\Big]\\
&=\mathbf{E}\Big[u(0, X_T,X(T))\|\mathcal{F}_{0}\Big]\\
&=\mathbf{E}\Big[f( X_T,X(T))\|\mathcal{F}_{0}\Big]=u(t,\eta,x)
\end{align*}
\end{proof}

\subsection{Viscosity Solution}

To state the reverse result we need to introduce the concept of a viscosity solution.

\begin{defn}
Let $V\in\mathbf{C}([0, T]×\mathbf{C})$. We say that $V$ is a viscosity sub-solution
of 
$$\frac{\partial}{\partial t}u(t)+\mathcal{A}_w(u(t))+c\cdot u(t)=0$$
$$u(T,\eta,x)=f(\eta,x)$$
with $\mathcal{A}_w$ as defined in Theorem \ref{Mohammedgenerator}, 
if, for every $\Gamma \in\mathbf{C}^{1,2}_{lip} ([0, T],\mathbb{C})\cap\mathcal{D}(S)$, and for $(t,\psi,)\in [0, T]\times\mathbf{C}$
satisfying $\Gamma\geq V$ on $[0, T]\times \mathbb{C}$ and $\Gamma(t,\psi) = V(t,\psi)$, we have
$$\frac{\partial}{\partial t}\Gamma(t)-SV+\Big[H(\Gamma(t))\cdot\nabla_x+\frac{1}{2}tr(\langle G, D^2(\cdot) G\rangle)\Big](\Gamma(t))\leq0$$
It is a supersolution if the analogous condition is met: $\Gamma\leq V$ on $[0, T] \times\mathbb{C}$, $\Gamma(t,\psi) = V(t,\psi)$ and we have
$$\frac{\partial}{\partial t}\Gamma(t)-SV+\Big[H(\Gamma(t))\cdot\nabla_x+\frac{1}{2}tr(\langle G, D^2(\cdot) G\rangle)\Big](\Gamma(t))\geq 0$$
A function $V$ is called a viscosity solution if it is simultaneously a sub-solution and a super-solution. 
\end{defn}
We are ready to state the reverse of Theorem \ref{FeynmannkacStef}.
\begin{thm}Suppose $f\in\mathcal{D}(\mathcal{A}_w)$ and $c:L^2([-r,0])\times\Real^d\to\Real_+$ bounded and Lipschitz continuous. The function
$$u(t,\eta,x):=\mathbb{E}_{(t,\eta,x)}\Big[f( ^\eta X_T,  ^\eta X_T(0))e^{-\int_t^T c( ^\eta X_s,  ^\eta X_s(0))ds}\Big]$$
is a viscosity solution of
$$\frac{\partial}{\partial t}u(t)+\mathcal{A}_w(u(t))-c\cdot u(t)=0$$
$$u(T,\eta,x)=f(\eta,x)$$
\end{thm}

In order to prove this result we need first a lemma that emphasizes the concept of markovianity of the solution of the SFDE $(X(t),X_t)$.

\begin{lem}
For $s,t\in [0,T]$ with $t\leq s$, we have
$$u(t,\eta,x):=\mathbb{E}_{(t,\eta,x)}\Big[u( ^\eta X_s,  ^\eta X_s(0))e^{-\int_t^s c( ^\eta X_s,  ^\eta X_s(0))ds} \Big]$$
\end{lem}

\begin{proof}
Let $s,t\in [0,T]$ such that $t\leq s$. Then the following equality holds:
\begin{align*}
u(s, X_s, X_s(0))=&\mathbb{E}\Big[f( X_T,  X_T(0))e^{-\int_u^T c( ^\eta X_s,  ^\eta X_s(0))ds}\Big\| X_u,X_u(0)\Big]\\
=&\mathbb{E}\Big[f( X_T,  X_T(0))e^{-\int_u^T c( ^\eta X_s,  ^\eta X_s(0))ds}\Big\| \mathcal{F}_u\Big]
\end{align*}
Since $( ^\eta X_u, ^\eta X_u(0))$ is Markovian, it follows from the Tower Property of the Conditional Expectation
\begin{align*}\mathbb{E}_{(t,\eta,x)}\Big[u( ^\eta X_s,  ^\eta X_s(0))e^{-\int_t^u c( ^\eta X_s,  ^\eta X_s(0))ds} \Big]&=\\
\mathbb{E}_{(t,\eta,x)}\Big[\mathbb{E}\Big[f( X_T,  X_T(0))e^{-\int_u^T c( X_s, X_s(0))ds}\Big\| \mathcal{F}_u\Big]e^{-\int_t^u c( ^\eta X_s,  ^\eta X_s(0))ds} \Big]&=\\
\mathbb{E}\Big[e^{-\int_u^T c( X_s, X_s(0))ds}e^{-\int_t^u c(X_s, X_s(0))ds}\mathbb{E}\Big[f( X_T,  X_T(0))\Big\| \mathcal{F}_u\Big]\Big\|\mathcal{F}_t \Big]&=\\
\mathbb{E}\Big[f( X_T,  X_T(0))e^{-\int_t^T c( X_s, X_s(0))ds}\Big\|\mathcal{F}_t\Big]&=u(t,\eta,x).
\end{align*}
\end{proof}

\begin{proof}[Proof of the Theorem]
We will be using the notation
$$\mathbb{E}_{\eta^x}[\cdot]:=\mathbb{E}[\cdot\|X_t=\eta, X_t(0)=x]$$
Let $\Gamma\in C^{1,2}_{lip}$ in the domain of the shift operator. For $0\leq t\leq t_1\leq T$, following Theorem 3.1 in \cite{MohammedSFDE}, we have that
$$\mathbb{E}_{\eta^x}\Big[e^{-\int_t^{t_1} c( X_s, X_s(0))ds}\Gamma(t_1,X_{t_1}, X_{t_1}(0))\Big]-\Gamma(t,\eta,x)$$
$$=\mathbb{E}_{\eta^x}\Big[\int_t^{t_1}e^{-\int_t^{u} c( X_s, X_s(0))ds}\Big(\frac{\partial}{\partial t}\Gamma(u)+\mathcal{A}_w(\Gamma(u))-c\cdot \Gamma(u)\Big)\Big]$$
where we have used the notation $\Gamma(u)=\Gamma(u,X_u, X_u(0))$.\\
From the previous lemma, for any $t_1\in[t,T]$ 
$$u(t,\eta,x)\geq\mathbb{E}_{(t,\eta,x)}\Big[u( ^\eta X_s,  ^\eta X_s(0))e^{-\int_t^s c( ^\eta X_s,  ^\eta X_s(0))ds} \Big]$$
By using $\Gamma\geq u$ the previous formula leads to
\begin{align*}
0&\geq \mathbb{E}_{\eta^x}\Big[e^{-\int_t^{t_1} c( X_s, X_s(0))ds} u(t_1,X_{t_1}, X_{t_1}(0))\Big]-u(t,\eta,x)\\
&\geq\mathbb{E}_{\eta^x}\Big[e^{-\int_t^{t_1} c( X_s, X_s(0))ds}\Gamma(t_1,X_{t_1}, X_{t_1}(0))\Big]-u(t,\eta,x)\\
&\geq \mathbb{E}_{\eta^x}\Big[\int_t^{t_1}e^{-\int_t^{u} c( X_s, X_s(0))ds}\Big(\frac{\partial}{\partial t}\Gamma(u)+\mathcal{A}_w(\Gamma(u))-c\cdot \Gamma(u)\Big)\Big]
\end{align*}
By dividing by $(t_1-t)$ and letting $t_1$ towards $t$ in the previous inequality, follows
$$\frac{\partial}{\partial t}\Gamma(t)-SV+\Big[H(\Gamma(t))\cdot\overline{\nabla}_x+\frac{1}{2}tr(\langle G, \overline{\Delta}(\cdot) G\rangle)\Big](\Gamma(t))\geq 0$$
In a similar fashion the other inequality is obtained. For any $t_1\in[t,T]$ 
$$u(t,\eta,x)\leq\mathbb{E}_{(t,\eta,x)}\Big[u( ^\eta X_s,  ^\eta X_s(0))e^{-\int_t^s c( ^\eta X_s,  ^\eta X_s(0))ds} \Big]$$
Now set $\Gamma\geq u$, and thus
\begin{align*}
0&\leq \mathbb{E}_{\eta^x}\Big[e^{-\int_t^{t_1} c( X_s, X_s(0))ds} u(t_1,X_{t_1}, X_{t_1}(0))\Big]-u(t,\eta,x)\\
&\leq\mathbb{E}_{\eta^x}\Big[e^{-\int_t^{t_1} c( X_s, X_s(0))ds}\Gamma(t_1,X_{t_1}, X_{t_1}(0))\Big]-u(t,\eta,x)\\
&\leq \mathbb{E}_{\eta^x}\Big[\int_t^{t_1}e^{-\int_t^{u} c( X_s, X_s(0))ds}\Big(\frac{\partial}{\partial t}\Gamma(u)+\mathcal{A}_w(\Gamma(u))-c\cdot \Gamma(u)\Big)\Big]
\end{align*}
By dividying by $(t_1-t)$ and letting $t_1$ towards $t$ in the previous inequality, follows
$$\frac{\partial}{\partial t}\Gamma(t)-SV+\Big[H(\Gamma(t))\cdot\overline{\nabla}_x+\frac{1}{2}tr(\langle G, \overline{\Delta}(\cdot) G\rangle)\Big](\Gamma(t))\leq 0$$
And the conclusion of the theorem follows.
\end{proof}

The following result implies the uniqueness of the solution.

\begin{thm}{Comparison principle}. Assume that $V_1(t,c)$ and $V_2(t,c)$ are both
continuous with respect to the argument $(t,c)$ and are respectively viscosity sub-solution and
super-solution of the FPDE with at most a polynomial growth. Then
$$V_1(t,c)\leq V_2(t,c) \forall (t,c)\in[0,T]\times C[-r,0]$$
\end{thm}

\begin{proof}
The proof follows the same argument as in \emph{Chang et al.} \cite{CPP06}.
\end{proof}

\section{The Feynman-Kac Formula - Boundary Value problem}

In this this section we develop the Feynman-Kac's formula for solution of SFDE constrain to a domain $D$.
Let us consider an open bounded domain $D$ of $\Real^d$ and the set of continuous functions $A=\mathbf{C}([-r,0],D)$ bounded uniformly by $M$. Let us consider the random time 
$$\tau^t_{\eta,x}:=\inf\{s\in[0,T]:( ^\eta X_s, ^\eta X(s))\in\partial (A\times D)\}\wedge t$$
and the stopped process
$$X^{\tau^t_{\eta,x}}(t)=\eta_0(t)+\int_0^{\tau^t_{\eta,x}\vee0} H(X_s,X(s))ds+\int_0^{\tau^t_{\eta,x}\vee0} G(X_s,X(s))dW(s)$$
where $\eta$ is defined as in the previous section and $H\in\mathbb{L}^{1,2}$ and $G\in\mathbb{L}^{1,2}$ are $\mathcal{F}_t$-adapted functions that satisfy the hypothesis of Lipschitz continuity with respect to both arguments that implies existence and uniqueness.\\\\
We underline that in this section we deal with the $C([-r,0],\Real^d)$ setting. In dealing with the infinitesimal generator we have to take some care respect to the $L^2$-setting\cite{MohammedSFDE}:\\
Let $\mathcal{L}(C)$ and $\mathcal{B}(C)$ be the space of bounded linear functionals $\Phi : C\to\Real$ and bounded bilinear functionals $\tilde\Phi : C\times C\to\Real$, of the space C, respectively. They are equipped with the operator norms which will be, respectively, denoted by $\norm{\cdot}_\mathcal{L}$ and $\norm{\cdot}_\mathcal{B}$. With $\mathbf{1}_{[a,b]}(t):=\mathbf{1}_{[-r,0]\cap[a,b]}(t)$
$$F_n:=\{v\mathbf{1}_{\{0\}}:v\in\Real^n\}$$
We form the direct sum 
$$C\oplus F_n:=\{\phi+v\mathbf{1}_{\{0\}} | \phi\in C,v\in\Real^n\}$$
and equip it with the norm $\norm{\cdot}$ defined by
$$\norm{\phi+v\mathbf{1}_{\{0\}}}:=\sup_{t\in[-r,0]}\phi(t)+|v|\quad \phi\in C,v\in\Real^n$$
Note that for each sufficiently smooth function $\Phi:C\to\Real$, its first order Fr\'echet derivative $D\Phi(\phi)\in\mathcal{L}(C)$ has a unique and continuous linear extension $\overline{D\Phi(\phi)}\in\mathcal{L}(C\oplus F_n)$. Similarly, its second order Fr\'echet derivative $D^2\Phi(\phi)\in\mathcal{B}(C)$ has a unique and continuous linear extension $\overline{D^2\Phi(\phi)}\in\mathcal{B}(C\oplus F_n)$.\\
For a Borel measurable function $\Phi:C\to\Real$, we also define the Shift Operator
$$\Gamma_t(\Phi)(\phi):=\Phi(\tilde{\phi_t}),$$
where for each $\phi\in C$ and $t\geq0$ $\tilde\phi:[-r,\infty)\to\Real^n$ is defined by
$$\tilde\phi(t):=
\begin{cases}
\phi(0)&t>0\\
\phi(t)&t\in[-r,0]
\end{cases}
$$
and define the operator
$$S(\Phi)(\phi):=\lim_{t\to0}\frac{1}{t}\Big[\Gamma_t(\Phi)(\phi)-\Phi(\phi)\Big]$$
whose domain $\mathcal{D}(S)$ is defined as the set of functions for which the limit exists.

\begin{thm}\label{Thm-InfinitesimalGenerator}
Let us suppose that $\Phi\in \mathbf{C}([0,T]\times \mathbf{C})$ satisfies the smoothness condition, $\Phi\in \mathbf{C}^{1,2}_{Lip}([0,T]\times \mathbf{C})$ and $\Phi\in\mathcal{D}(S)$. Let $\{X_s, s\in[t,T]\}$ be the $C$-valued Markov solution defined above with initial data $(t,\phi_t)\in[0,T]\times C$. Then
\begin{align*}
\mathcal{A}_w\Phi(t,\phi)=&\lim_{\epsilon\to0}\frac{\mathbb{E}[\Phi(t+\epsilon,X_{t+\epsilon})]-\Phi(t,\phi_t)}{\epsilon}\\
=& \frac{\partial }{\partial t}\Phi(t,\phi_t)+S(\Phi)(t,\phi_t)+\overline{D\Phi(t,\phi_t}(H(t,\phi_t)\mathbf{1}_{\{0\}})\\
&+\frac{1}{2}\sum_{j=1}^m\overline{D^2\Phi(t,\phi_t}(G(t,\phi_t)(\mathbf{e}_j)\mathbf{1}_{\{0\}},G(t,\phi_t)(\mathbf{e}_j)\mathbf{1}_{\{0\}})
\end{align*}
where $\mathbf{e}_j, j=1\ldots n$ is the $j-$th vector of the standard basis in $\Real^m$. 
\end{thm}

Let us confine ourself to the class of quasi-tame functions \cite{MohammedSFDE}.

\begin{defn}
A function $\phi:\mathbf{C}([-r,0],\Real^m)\to\Real$ is \emph{quasi-tame} if there is an integer $k>0$, $\mathbf{C}^\infty$ maps $f_j:\Real^m\to\Real^m$, $h:\Real^{n\times k}\to\Real$ and piece-wise $\mathbf{C}^1$ function $g_j:[-r,0]\to\Real$, with $1\geq j\geq k-1$, such that for all $\eta\in\mathbf{C}([-r,0],\Real^m)$ we have
$$\phi(\eta)=h\Big((\int_{-r}^0f_j(\eta(s))g_j(s)ds)_{j=1}^{k-1};\eta(0)\Big)$$
\end{defn}

\begin{thm}
Suppose $\psi\in L^2(\Omega,\mathbf{C})$ and the operator $\mathcal{A}_q$ defined in Theorem \ref{Thm-InfinitesimalGenerator} applied to the class of quasi-tame functions. Then the martingale problem for $(\mathcal{A}_q,\psi)$ is well posed. 
\end{thm}

\begin{lem}
Suppose $f\in\mathcal{D}(\mathcal{A}_w)$ and $\mathbb{E}[\tau_{x,\eta}^D]<\infty$. If $u$ solves classically
$$\mathcal{A}_w(u(x,\eta))-c(x,\eta)u(x,\eta)=f(x,t)\quad (\eta,x)\in A\times D$$
$$u(\eta,x)=g(\eta,x)\quad (\eta,x)\in\partial (A\times D)$$
where $g(\eta,x)$ belongs to the class of quasi tame functions, then
$$u(\eta,x)=-\mathbb{E}_{(0,\eta,x)}\Big[\int_0^{\tau_{x,\eta}^t}f( ^\eta X_{s},  ^\eta X(s))e^{-\int_0^{s}c( ^\eta X_{u},  ^\eta X(0))du}\Big]$$
$$+\mathbb{E}_{(0,\eta,x)}\Big[g( ^\eta X_{\tau^t_{\eta,x}},  ^\eta X(\tau^t_{\eta,x}))e^{-\int_0^{\tau^t_{\eta,x}}c( ^\eta X_{u},  ^\eta X(u))du}\Big]$$
\end{lem} 
\begin{proof}
 
The proof can be done following the proof in \cite{MF85}, Theorem 2.1 page 127, using the It\^{o} formula for quasi-tame functions.
\end{proof}
Similarly as in the previous section, let us suppose that the \textbf{Assumption A} is satisfied. Consider the following SFDE:

Set $\eta_0:[-r,t]\to\Real^d$, with $\eta_0(s)=\eta_0(0)$ for $s\geq 0$ and $t^{t,\eta}_s=t-s$. 

If $s\in[-r,t]$ 
\begin{equation*}
X^{t,\eta}(s)=\eta_0(s)+\int_0^{s\vee0} H(t^{t,\eta}_s,X^{t,\eta}_u,X^{t,\eta}(u))ds+\int_0^{s\vee0} G(t^{t,\eta}_s,X^{t,\eta}_u,X^{t,\eta}(u))dW(u).
\end{equation*}
It is possible to define in accordance with the previous section a Markov family $$\Big(t^{t,\eta}_s, ^\eta X_s, ^\eta X_s(0)\Big)\in\Real\times L^2([-r,0],\Real^d)\times\Real^d$$
In this case the infinitesimal generator is given by $\mathcal{A}^-_w$, defined as
\begin{align*}
\widetilde{\mathcal{A}}^-_w\Phi(t,\phi)=& -\frac{\partial }{\partial t}\Phi(t,\phi_t)+S(\Phi)(t,\phi_t)+\overline{D\Phi(t,\phi_t}(H(t,\phi_t)\mathbf{1}_{\{0\}})\\
&+\frac{1}{2}\sum_{j=1}^m\overline{D^2\Phi(t,\phi_t}(G(t,\phi_t)(\mathbf{e}_j)\mathbf{1}_{\{0\}},G(t,\phi_t)(\mathbf{e}_j)\mathbf{1}_{\{0\}})
\end{align*}

The following theorem holds:
\begin{thm}\label{thm31}Suppose $f\in\mathcal{D}(\mathcal{A}_w)$. If $u$ solves weakly
$$\frac{\partial}{\partial t}u(t)=\mathcal{A}_w(u(t))+c(x,\eta)u(x,\eta)\quad (t,\eta,x)\in[0,T]\times A\times D$$
$$u(0,\eta,x)=f(\eta,x)\quad (\eta,x)\in A\times D$$
$$u(t,\eta,x)=g(t,\eta,x)\quad (\eta,x)\in\partial (A\times D)$$
then
\begin{align*}
u(t,\eta,x)=&\,\mathbb{E}_{(0,\eta,x)}\Big[f( ^\eta X_t,  ^\eta X_t(0))\mathbf{1}_{\{\tau^t_{\eta,x}=t\}}e^{-\int_0^{t}c( ^\eta X_{u},  ^\eta X(u))du}\Big]\\
&+\mathbb{E}_{(0,\eta,x)}\Big[g(\tau^t_{\eta,x}, ^\eta X_t,  ^\eta X_t(0))\mathbf{1}_{\{\tau^t_{\eta,x}\not= t\}}e^{-\int_0^{\tau^t_{\eta,x}}c( ^\eta X_{u},  ^\eta X(u))du}\Big]
\end{align*}
\end{thm}

The reverse holds in case of viscosity solutions of the FPDE.

\begin{thm}
\begin{align*}
u(t,\eta,x)=&\,\mathbb{E}_{(0,\eta,x)}\Big[f( ^\eta X_t,  ^\eta X_t(0))\mathbf{1}_{\{\tau^t_{\eta,x}=t\}}e^{-\int_0^{t}c( ^\eta X_{u},  ^\eta X(u))du}\Big]\\
&+\mathbb{E}_{(0,\eta,x)}\Big[g(\tau^t_{\eta,x}, ^\eta X_t,  ^\eta X_t(0))\mathbf{1}_{\{\tau^t_{\eta,x}\not= t\}}e^{-\int_0^{\tau^t_{\eta,x}}c( ^\eta X_{u},  ^\eta X(u))du}\Big]
\end{align*} is a viscosity solution of the system
$$\frac{\partial}{\partial t}u(t)=\mathcal{A}_w(u(t))+c(x,\eta)u(x,\eta)\quad (t,\eta,x)\in[0,T]\times A\times D$$
$$u(0,\eta,x)=f(\eta,x)\quad (\eta,x)\in A\times D$$
$$u(t,\eta,x)=g(t,\eta,x)\quad (\eta,x)\in\partial (A\times D)$$
\end{thm}


\section{First Exit Time Probability for SFDE}


Let us denote by $\tau_D$ the first exit time of $X^{x,\eta}(t)$, where $X^{x,\eta}(t)$ is the solution of the SFDE with initial conditions $x,\eta$. Let $\mathbf{Q}(t,x,\eta)$ be the probability that $X^{x,\eta}$ starting from $x,\eta$ did not exit the domain $D\subset\Real^d\times C([-r,0],\Real^d)$ before $t$, i.e.
$$\mathbf{Q}(t,\eta,x)=1-\mathbf{P}_{x,\eta}\Big(\tau_A<t\Big)$$

\subsection{First Exit Time probability as a Viscosity Solution }

Let us consider the process solution of the SFDE:
$$X(t)=\eta_0(t)+\int_0^{t\vee0} H(X_s,X(s))ds+\int_0^{t\vee0} G(X_s,X(s))dW(s).$$

Throughout this section we impose the following stringent hypothesis about the FPDE:
\begin{equation}\label{Prob_exit}\begin{cases}
\frac{\partial}{\partial t}u(t)=\mathcal{A}_w(u(t))& (t,\eta,x)\in[0,T]\times D\\
u(0,\eta,x)=1& (\eta,x)\in D\\
u(t,\eta,x)=0& (t,\eta,x)\in]0,T[\times\partial D
\end{cases}\end{equation}

\textbf{Hypothesis B}: the variational problem (\ref{Prob_exit}) belongs to 
$$u\in \mathbf{C}^{0}\Big([0,T];\mathbf{C}^2([0,T],D)\cap\mathbf{C}^2(\overline{D})\Big).$$

It is then possible to state the following result:
\begin{thm}
Under \textbf{Hypothesis B}, the function
$$\mathbf{Q}(t,\eta,x)=1-\mathbf{P}_{x,\eta}\Big(\tau_A<t\Big)$$
is a Viscosity solution of the problem:
\begin{equation*}\begin{cases}
\frac{\partial}{\partial t}u(t)=\mathcal{A}_w(u(t))& (t,\eta,x)\in[0,T]\times D\\
u(0,\eta,x)=1& (\eta,x)\in D\\
u(t,\eta,x)=0& (t,\eta,x)\in]0,T[\times\partial D
\end{cases}\end{equation*}
\end{thm}



\section{Applications}

\subsection{Movement of E.coli}

The motion of \emph{E. coli} bacteria is characterized by a sequence of run and tumble events \cite{Berg}. During a run the flagella of the bacteria rotate counter-clockwise, form a bundle and propel the cell in a more or less straight line. If the flagella rotate clockwise, the bundle opens and the bacteria randomly change their angle of motion without forward propagation (tumble). This bacterium belongs to the species that due to their small size are unable to sense chemoattractant gradients reliably. The evolution has then developed a history-dependent strategy for search of food, namely on the use of the memory of previous measurements of chemical concentrations. In this way the bacterium is able to infer whether the swimming is done up or down a chemical gradient.

\subsubsection{Tumble Probability}
Let us suppose that at time $t=0$, the bacterium is at position $x_0$. The bacterium is characterized by an internal-state variable, say $\Lambda$, which modulates the \emph{turning probability}: given a level of saturation $\tau_0$, the bacterium tumbles when the process $\Lambda$ hits the level $\tau_0$. The process $\Lambda(t)$ satisfies the following system of SFDE's:
let $0\leq t\leq T,$
\begin{equation*}
\begin{aligned}
\zeta (t)&=F(c(X (t),t),\theta(t),\dot{\zeta}(t)),\\
\Lambda (t)&=\Lambda _0+\int_0^t\lambda(s,\zeta (s),\zeta _s,\Lambda _s,\theta (s))ds+\int_0^t\sigma(s,\zeta (s),\zeta _s,\Lambda _s,\theta (s))dW (s).\\
\end{aligned}
\end{equation*}
where $c(x,t)$ is the concentration of attractant at $(x,t)$, $\theta(t)$ is the direction along which the bacterium swims (it is constant between two jumps), and $X(t)$ is the position, which satisfies 
$$X(t)=X(\tau_1)+(t-\tau_1)\theta,\quad t\in[\tau_1,\tau_2).$$
For $t\in[-r,0]$ we assume the initial processes:
$$(x_0(t))_{t\in[-r,0]},\quad(\theta(t))_{t\in[-r,0]},\quad(\Lambda_0(t))_{t\in[-r,0]},\quad(\zeta_0(t))_{t\in[-r,0]}.$$

The function 
$$\mathbf{Q}(t,\eta,x)=1-\mathbf{P}_{x,\eta}\Big(\tau_A<t\Big)$$
is a the distribution for the length of a \emph{run}:
\begin{equation*}\begin{cases}
\frac{\partial}{\partial t}u(t)=\mathcal{A}_w(u(t))& (t,\eta,x)\in[0,T]\times D\\
u(0,\eta,x)=1& (\eta,x)\in D\\
u(t,\eta,x)=0& (t,\eta,x)\in]0,T[\times\partial D
\end{cases}\end{equation*}

%
%

\subsection{Infinite-Dimensional Black-Scholes Equation with Hereditary Structure}
The price of options in the continuous time $(B,S)$-market has been a subject of extended research in recent years. Let consider a slight modification of the model proposed in \cite{CY}. 
The idealized Black-Scholes $(B,S)$-market often consists of an account $(B(t))_{t\in[0,T]}$ and the stock $(S(t))_{t\in[0,T]}$. The equation for the evolution of the prices of these two financial products are given by the following system of SFDE's: Let us suppose that the solution process $B(\phi)$ satisfies the following equality
$$B(t)=\phi(0)e^{\int_0^tr(s)ds}$$
where $(\phi(t))_{t\in[-r,0]}$ is the initial condition. 
Let $T>0$ be the expiration time for the European options considered in this example. Assume that the stock price $(S(t))_{t\in[-r,T]}$ satisfies the following nonlinear stochastic functional differential equation:
$$\frac{dS(t)}{S(t)}=f(S_t)dt+g(S_t)dW(t),\quad t\in[0,T]$$
with initial price function $\psi$.
Using classical arguments for Trading Strategy and Equivalent Martingale Measure, the pricing formula $V:[0,T]\times C[-r,0]\to\Real$ satisfies the following expression:
$$V(t,\psi)=\mathbb{E}^t_{\psi}\Big[e^{-\int_t^Tr(s)ds}\lambda(S_T)\Big]$$
\begin{thm}
Assume that $V(t,\psi)\in\mathbf{C}^{1,2}_{\textrm{Lip}}([0,T]\times \mathbf{C})\cup\mathcal{D}(S)$ and that the market is \emph{self-financial}, then $V(t,\psi)$ satisfies the following equation:
\begin{align*}
r(t)V(t,\psi)=&\frac{\partial }{\partial t}V(t,\psi)+S(V)(t,\psi)+\overline{DV(t,\psi)}(r(t)\psi(0)\mathbf{1}_{0})\\
&+\overline{D^2V(t,\psi)}(\psi(0)g(\psi)\mathbf{1}_{0},\psi(0)g(\psi)\mathbf{1}_{0})\quad (t,\psi)\in[0,T)\times \mathbf{C}\\
V(T,\psi)=&\lambda(\psi)\quad\quad  \psi\in\mathbf{C}
\end{align*}
  And the reverse holds in the sense of a viscosity solution.
\end{thm}

%

\end{document}